\documentclass[12pt]{article}
\usepackage{CormorantGaramond}
\usepackage[english]{babel}
\usepackage{calligra}
\usepackage[T2A, T1]{fontenc}
\usepackage[utf8]{inputenc}
\usepackage{textcomp}
\usepackage[numbers]{natbib}

\usepackage[left=4cm, right=10cm, top=4cm]{geometry}
\usepackage{lmodern}
\usepackage{amsmath}
\usepackage{url}
\usepackage[bottom]{footmisc}
\usepackage{environ}
\usepackage{nicefrac}
\usepackage{subcaption}
\usepackage{amsfonts,amssymb,amsbsy}
\usepackage{mathtools}
\usepackage{nccmath}
\usepackage{multirow}
\usepackage{bbm}
\usepackage{comment}
\usepackage{lipsum}
\usepackage{microtype} 
\usepackage{placeins}
\usepackage{xspace}
\usepackage{tocbibind}
\usepackage{graphicx}
\usepackage{float}
\usepackage{array}
\usepackage{tabularx}
\usepackage{xcolor}
\usepackage{color}
\usepackage{hyperref}
\usepackage{acronym}
\usepackage{fullpage}
\usepackage[normalem]{ulem} 
\usepackage{makeidx}
\usepackage{wrapfig}
\usepackage{tikz}
\usepackage{amsthm, amssymb}
\usepackage{algpseudocode}
\usepackage{algorithm}

\theoremstyle{definition}

\newtheorem{dfn}{Definition}

\newtheorem{lem}{Lemma}
\newtheorem{thm}{Theorem}

\newtheorem{rem}[thm]{Remark}

\newcommand{\N}{\mathbb{N}}
\newcommand{\R}{\mathbb{R}}

\makeatletter
\definecolor{myblue}{RGB}{0, 128, 128}

\hypersetup{colorlinks=true,linkcolor=myblue,urlcolor=myblue,citecolor=myblue,}
\let\mytagform@=\tagform@
\def\tagform@#1{\maketag@@@{\color{myblue}(#1)}}
\makeatother

\makeindex
\title{On Unique Ergodicity Of Coupled AIMD Flows}
\author{Pietro Ferraro, Jia Yuan Yu, Ramen Ghosh, Syed Eqbal Alam, \\ Jakub Marecek, Fabian Wirth, and Robert Shorten}

\begin{document}
\maketitle
\begin{abstract}
The AIMD algorithm, which underpins the Transmission Control Protocol (TCP) for transporting data packets in communication networks, is perhaps the most successful control algorithm ever deployed. Recently, its use has been extended beyond communication networks, and successful applications of the AIMD algorithm have been reported in transportation, energy, and mathematical biology. A very recent development in the use of AIMD is its application in solving large-scale optimization and distributed control problems without the need for inter-agent communication. In this context, an interesting problem arises when multiple AIMD networks that are coupled in some sense (usually through a nonlinearity). The purpose of this note is to prove that such systems in certain settings inherit the ergodic properties of individual AIMD networks. This result has important consequences for the convergence of the aforementioned optimization algorithms. \textcolor{black}{The arguments in the paper also correct conceptual and technical errors in \cite{Syed2020}}.         
\end{abstract}

\section{Introduction}
The issue of ergodicity in stochastic systems under the influence of feedback has recently emerged as a topic of interest in the control engineering community. While traditional interest in this topic has been driven by the mathematical community, new and recent interest is driven by very practical considerations that arise as society embraces new disruptive business paradigms such as the sharing economy and the circular economy \cite{Fioravanti2019}. In particular, in the design of sharing economy systems, one is interested in allocating (sharing) resources in some manner. To do this, we are interested in using randomized or stochastic algorithms (to, for example, avoid issues such as flapping in routing or other load balancing systems, refer \cite{Crisostomi2020} for details). The issue of ergodicity arises naturally in such a context. Roughly speaking, ergodicity plays the role of {\em independence of initial conditions} in the study of deterministic ordinary differential equations. Ergodic behaviour, namely this independence of initial conditions, is thus fundamental, both from the perspective of the design of systems where, for example, reproducible (in a statistical sense) simulations are essential, and also from the point of view of issuing economic contracts where reproducible levels of service are also necessary.\newline 

As we have mentioned, while the study of ergodic behaviour is already a mature topic in mathematics, the study of ergodicity as it arises in the context of the sharing economy is bringing new perspectives to this old problem. For example, the preservation of ergodicity in the design of feedback systems appears to be a new problem in control theory \cite{Fioravanti2019}, as is the study of  ergodic behaviour in several feedback systems that are themselves coupled together. Such systems arise, for example, when a group of agents share more than one resource, and the utility of each allocation varies in a nonlinear manner amongst agents. This problem, which considers the effect of nonlinear coupling in systems that otherwise exhibit ergodic behaviour, appears to be unexplored in the stochastic systems community, and it is precisely this problem that we shall start to address in this paper. More precisely, we are interested in exploring the use of the Additive Increase Multiplicative Decrease (AIMD) algorithm for allocating multiple resources amongst agents in a distributed manner. The AIMD algorithm is one of the most widely deployed and successful feedback control algorithms currently utilised in society. From its beginnings as a purely distributed congestion management algorithm in computer networking, it has proved its utility across multiple domains, including charging electric vehicles, smart grid applications, and other smart city applications \cite{Marecek2015, Marecek2016}. Recently, the AIMD algorithm has also been shown to exist in the natural and biological world [2]. An essential question in the study of AIMD networks concerns their \emph{ergodicity}. The ergodic behaviour of elementary AIMD networks was first reported in [4] using tools from the study of iterated functions systems \cite{Barnsley1988}. This short paper aims to show that networks of coupled AIMD systems also exhibit the ergodic property under very mild assumptions. To the best of our knowledge, this is the first such result to appear in the literature.\newline

\textcolor{black}{{\bf Comment} (connection to \cite{Syed2020}): An important contribution of this note is to correct a modelling error in \cite{Syed2020}. The paper \cite{Syed2020} extends the finite averaging concepts developed in \cite{Wirth2019} to the case of coupled resources. However, the contraction analysis presented therein is erroneous. The present note corrects this error by enforcing instants in time where all flows are synchronized, thereby facilitating analysis using contraction arguments.}  

\section{Preliminary results}
\subsection{Notations, definitions, and terminology}

The vector space of real column vectors with $n$ entries is denoted by $\mathbb R^n$ with elements $x =
\begin{bmatrix}
x_{1} & \ldots & x_{n}
\end{bmatrix}^\top$, where $x^\top$ denotes the transpose of $x$. 
We denote $e := [1 \quad 1 \quad \ldots \quad 1]^\top \in \mathbb R^n$ and the $i$'th standard basis vector by $e_i$. The identity matrix is denoted by $I$. The norm we use on $\mathbb R^n$ is the $1$-norm, defined by $\left\|x\right\|_1 = \sum_{i=1}^n |x_i|$.
The standard simplex $\Sigma$ in $\mathbb{R}^{n}$ is defined by $\Sigma= \left \{ x = (x_{1},x_{2},\ldots,x_{n}) \in \mathbb{R}^{n}_+ : \sum_{i=1}^{n} x_{i} =1 \right \}$. We will write $\Sigma_n$ if we want to make the dimension of the simplex explicit.

\subsection{A primer on the AIMD algorithm and AIMD matrices}
The AIMD algorithm is a {\em distributed} feedback control algorithm that is used to allocate a resource amongst a network of agents competing for that resource. The best known example of its application is in internet congestion control whereby multiple agents compete for a limited bandwidth \cite{Jacobson1988}, \cite{Chiu1989}. In this context, the AIMD algorithm is the basis for the transmission control protocol (TCP) that dictates the movement of packets in almost all networking applications.  
In such situations, agents in the network are assumed to operate the AIMD algorithm which, in turn, is built from two distinct components (phases), {\em the additive increase} (AI) phase and the {\em multiplicative decrease} (MD) phase. As the AIMD algorithm is described in great detail in many textbooks and papers \cite{Wirth2019}, \cite{Corless2016}, we shall not repeat this detailed discussion here. Rather, we give a very brief flavour of the algorithm and focus on reminding the reader of the mathematical modelling of a network of AIMD agents, as well as the relevant results that shall be of use in the paper. We begin with a high-level description of the AIMD phases, and the mechanism by which an agent switches between these phases. \newline  

{\bf The AI phase:} During the AI phase, agents gradually increase the share of the resource that they have acquired. Usually, but not always, agents increase their share in a manner that is proportional to the time that has elapsed since the last multiplicative phase. The positive constant of proportionality is denoted by $\alpha_i$ for agent $i$. \newline 

{\bf The MD phase: } At some instant in time, the cumulative allocation of the resource allocated to the agents will equal or exceed the available amount of resource. In the AIMD nomenclature, this event is called a {\em capacity event}. At this instant in time, a subset (or perhaps all) of the agents are informed of the capacity event, and all such informed agents reduce their allocation of resource in a multiplicative fashion. For example, in internet applications, it is common for agents to reduce their allocation by $50\%$. The multiplicative decrease factor by which agents reduce their allocation is denoted $\beta_i$ where $\beta_i \in (0,1)$. In what follows, it is very convenient to encode agents that may or may not be informed of the $k$'th capacity event by defining an event-driven  MD factor $\beta_i(k) \in \{\beta_i,1\}$; namely when an agent responds to a capacity event we set $\beta_i(k) = \beta_i$; if it does not respond to congestion we set $\beta_i(k) = 1$. \newline

{\bf AIMD drop policies: } We have not discussed the manner in which agents are informed of, and respond to, a capacity event; namely the mechanism by which an agent switches between the MD and the AI phases of the AIMD algorithm. In TCP, for example, agents deduce a capacity event when their packets are lost. Agents that lose packets automatically enter the MD phase. In other applications, all agents may be informed of capacity events and respond stochastically to these notifications according to some probability function (which we refer to as a {\em drop policy}). The ability to design these drop policies is of great utility in using the AIMD to solve certain optimization problems \cite{Wirth2019}.  \newline  

Together, this gives rise to a model of agent behaviour of the form:
\begin{eqnarray}
x_i(k+1) = \beta_i(k) x_i(k) + \alpha_i T(k),
\label{eq: basic aimd}
\end{eqnarray}
where $T(k)$ is the time between the $k$ and $(k+1)^{\text{th}}$ capacity event (notice that $T(k)$ is a function of $\alpha, \beta$ and the aggregate response of all agents to the capacity event), where $x_i(k)$ is the allocation of resource to agent $i$, and where the decision whether $\beta_i(k) = 1$ or $\beta_i(k) = \beta_i$ is taken according to a drop policy. As we have mentioned, the mathematical description of a network of agents operating the AIMD algorithm with a single constraint is thoroughly discussed in \cite[Section 1.2, 1.3] {Corless2016}. It is shown there that the dynamics of the set of $n$ users between the $k^{\text{th}}$ and $(k+1)^{\text{th}}$ capacity events can be described by the \emph{switched linear system}
\begin{align}\label{eq:aimd-dyna}
x(k+1)=A(k)x(k),
\end{align}
where $k$ enumerates the \emph{capacity events}, for $k=0,1,2,\dots,$ $x(k)$ is a 
vector with values in the simplex $\Sigma_n$, $A(k)$ is a non-negative column stochastic matrix belonging to a finite set $\mathcal{A}$ of  matrices that are indexed by a finite set $\mathcal S$ of indices, i.e.,
\begin{align}
\mathcal A=\{ A_{j} :j\in \mathcal S\}.
\end{align}

\begin{rem} Each of the $A_j$ matrices describes a different combination of agents that can respond to capacity events. 
\end{rem}

\begin{rem}
In (2), the selection of $A(k) \in \mathcal{A}$ is often probabilistic and governed by a drop policy. To be more precise, the matrix invoked at the capacity event $k$ is determined in many applications in a stochastic manner. When these probabilities depend on the state $x(k)$ we say that the system is {\em place-dependent}. 
\end{rem}

The matrices in the set $\mathcal{A}$ are constructed as follows. Let  $\alpha_i>0$, $0\leq\beta_i < 1$, $i\in \{1,2,\dots, n\}$ and $\beta_i(k) \in \{\beta_i,1\}$ for all $k\in\N$. Then
\begin{align}\label{eq:growth-decrs}
\alpha=\begin{bmatrix}\alpha_1&\cdots&\cdots&\alpha_n\end{bmatrix}^{\top} \text{ and } \beta(k)=\begin{bmatrix}\beta_1(k)&\cdots&\cdots&\beta_n(k)\end{bmatrix}^{\top}   
\end{align}
 define the growth and decrease vector, respectively.    
Then, for all $k$:
\begin{align}\label{eq:aimd-matrix-def}
A(k)&:= \text{diag}\left(\beta(k)\right)+ \left(e^{\top}\alpha\right)^{-1}\alpha\left(e^{\top}-\beta(k)^{\top}\right)\\ &=\begin{bmatrix}\beta_1(k)&\cdots&\cdots&0\\0&\beta_2(k)&\cdots&0\\\vdots&\vdots&\ddots&\vdots\\0&\cdots&\cdots&\beta_n(k)\end{bmatrix} \nonumber  +
\frac{1}{\sum\limits_{i=1}^{n}\alpha_i} \begin{bmatrix}\alpha_1\\ \vdots\\\vdots\\\alpha_n\end{bmatrix}\begin{bmatrix}1-\beta_1(k)&\cdots&\cdots&1-\beta_n(k)\end{bmatrix}.
\end{align}

We say that a matrix of the form \eqref{eq:aimd-matrix-def} is an {\em AIMD matrix}.
The interested reader can refer to \cite{Corless2016}, for a detailed explanation of the connection between (\ref{eq: basic aimd}), (\ref{eq:aimd-dyna}) and (\ref{eq:aimd-matrix-def}).



\begin{rem}
AIMD matrices are column-stochastic. The vector $e^\top = [1,\ldots,1]$ is a left-eigenvector corresponding to the Perron eigenvalue $1$. If $\lambda\ne 1$ is any other eigenvalue then $\left|\lambda\right|<1$ see \cite[Lemma 2.1]{Corless2016}. For a detailed discussion of the properties of AIMD matrices, see \cite{Corless2016}.
\end{rem}
\begin{dfn}(Drop matrix \cite{Corless2016})
An AIMD matrix $A$ is called a \emph{full-decrease} drop matrix if $\beta_i(k) = \beta_i <1$, for all $i\in\{1,2,\dots,n\}$.
\end{dfn}
\begin{dfn}(Contraction on invariant subspace \cite{Corless2016})
Let $\mathcal{E}$ be an invariant subspace of  $M \in \R^{n\times n}$. Let $\left\|\cdot\right\|$ be a norm on $\mathbb R^n$. The matrix $M$ is called a contraction on $\mathcal{E}$ (with respect to $\left\|\cdot\right\|$), if 
\begin{align*}
\left\|Mv\right\|\le \left\|v\right\| \quad \text{ for all } v\in \mathcal{E}.
\end{align*}
\end{dfn}
When measured with a suitable norm, if the difference between the state vectors becomes smaller due to the AIMD matrix action, we say that the matrix has a contractive effect. For more on contraction and related results on AIMD matrices, the reader is encouraged to see \cite[Section 3.1, Lemma 3.5, Lemma 3.7, Lemma 3.8]{Corless2016}.

\subsection{AIMD algorithm for a single resource}

The dynamical system describing a networks of $n$ agents, each operating the AIMD algorithm, is given by the stochastic difference equation \eqref{eq:aimd-dyna} \cite[cf.]{Corless2016}. We now consider a state-dependent AIMD model \cite[Chapter 7]{Corless2016}, which we shall extend to the case of multiple resources. To describe this dependency more precisely, let $\{p_j: j\in \mathcal S\}$ be a set of Lipschitz continuous probability functions from the simplex $\Sigma_n$ into the closed interval $[0,1]$ that satisfy
\begin{align}
\label{eq:probfuncs}
\sum\limits_{j\in \mathcal S} p_{j}\left(z\right)=1 \quad \text{ for all } z\in \Sigma_n.
\end{align}
Here $p_{j}(z)$ is the probability that 
the matrix $A_{j}$ occurs when the state or the share vector is $z\in\Sigma_n$. The stochastic AIMD model is thus given by
\begin{subequations}
\label{eq:AIMDmodel}
\begin{align}
\label{eq:Ak-stoch}
x(k+1) &= A(k) x(k), &x(0) \in \Sigma_n, \\
\mathbb{P}\left(A(k)=A_j\mid x(k)=z\right)&=p_{j}(z), 
&k=0,1,2,\dots,   
\label{eq:ocuurence-mat}
\end{align}
\end{subequations}

with the understanding that the random variables $A(k)$ conditioned on $x(k)$, $k\in \N$, are mutually independent.
Notice that
\eqref{eq:AIMDmodel} defines a stochastic AIMD algorithm and constitutes a specific instance of an \emph{iterated function systems} (IFS, cf. \cite{Barnsley1988}). This algorithm is referred to as the \emph{state-dependent AIMD} model. Furthermore, a Markov process is defined on the simplex $\Sigma_n$ by this model whose state-transitions probabilities are given by, see \cite[Chapter 9]{Corless2016},
\begin{align}\label{eq:transition-prob-aimd}
\mathbf{P}(z, \mathcal G)= \mathbb{P}\left(x(k+1)\in \mathcal G\mid x(k)=z\right)=\sum\limits_{j: A_{j}z\in \mathcal G} p_{j}(z) \quad \text{ for all } z\in \Sigma_n, \text{ and for any event } \mathcal G.   
\end{align}
Clearly, the right side of the Equation \eqref{eq:transition-prob-aimd} does not depend on the time $k$, hence the transition probabilities are homogeneous in time.

\subsection{Invariant measures and ergodicity}
\label{sec:preinv}

Conditions for the existence of a unique, invariant, and attractive probability distribution for Markov chains with place-dependent probabilities have been studied extensively over the past decades.

\begin{dfn}(Ergodicity)
The AIMD model has the ergodic property if it has a unique invariant distribution $\mu$ and if for every initial condition $x(0)=x_0$ with $x_0\in \Sigma$ and for every continuous function $\phi: \Sigma\to \mathbb R$,
\begin{align}\label{eq:ergo}
 \lim_{k\to \infty} \frac{1}{k+1} \sum\limits_{j=0}^{k} \phi(x(j))=\int \phi(z)\mu(\mathrm{d}z)   \quad \text{ almost surely } \mathbb{P}_{x_0}.
\end{align}
\end{dfn}
The left side of \eqref{eq:ergo} is the time average of the function $\phi$ applied to the vector $x(k)$ starting at $x_0$, whereas the right side is the ensemble average of $\phi$ with respect to the invariant distribution. Thus, the ergodic property equals the sequence average and ensemble average, allowing us to expectation with a sequence average. The right side does not depend on the starting state $x(0)=x_0$, hence it also ensures that the long-term average is independent of the initial state. A numerical simulation of the average will ultimately converge to the true average.

To show that our model described above is uniquely ergodic, we will use the following result which is a special case of a result of \cite{Barnsley1988}: 
\begin{thm}\cite[Theorem 2.1]{Barnsley1988}\label{thm:barns}
Consider a state-dependent AIMD model described by \eqref{eq:AIMDmodel} where $\mathcal S$ is finite and each probability function $p_{j}, j\in\mathcal  S$ is Lipschitz. Suppose that there exist $r<1$ and $\delta>0$ such that the following two conditions hold for all $x,y\in \Sigma_n$:
\begin{itemize}
\item[(a)] \begin{align}
\sum\limits_{j\in \mathcal S} p_{j}(x) \left\|A_{j}(x-y)\right\|\le r\left\|x-y\right\|
\end{align} 
\item[(b)] 
\begin{align}
\sum\limits_{j\in \mathcal C(x,y)} p_{j}(x)p_{j}(y)\ge \delta,
\end{align}
where $\mathcal C(x,y)=\{j \in \mathcal S: \left\|A_{j}(x-y)\right\|\le r\left\|x-y\right\|\}$.    
\end{itemize}
Then the AIMD model has a unique, invariant, and attractive probability distribution.
\end{thm}

\section{A preamble: AIMD and optimization}
The principal motivation for the present paper stems from the observation that an appropriately re-purposed AIMD algorithm can be used to solve large-scale optimization problems in a manner that is cheap from a communications-complexity perspective \cite{Corless2016}. We consider $n$ agents competing for a shared resource and having given utility functions $f_i:[0,1] \to \R_+$, which are $\mathcal{C^1}$ and strictly convex. The aim is to arrive at an optimal distribution of the resource, i.e., to find the optimal point for the optimization problem
\begin{equation*}
\label{eq:network-utility-opt}
    \min_{x\in \Sigma_n} \sum_{i=1}^n f_i(x_i).
\end{equation*}
Agents can apply the place-dependent AIMD algorithm with a low communication complexity and still converge, in a particular sense defined in Section~\ref{sec:preinv}. In such applications, we assume that all agents are informed of capacity events, and choose whether to respond to these events in a probabilistic manner according to an appropriate drop policy. 

In association with \eqref{eq:aimd-dyna}, let us define the long-term averages.
\begin{equation}
    \label{eq:longav}
    \overline{x}_i(k) = \frac{1}{k+1}\sum_{\ell = 0}^k x_i(\ell).
\end{equation}
In \cite{Wirth2019},
Wirth et al. consider the case, where agents respond with probability:
\begin{equation}
\label{eq: probOpt}
	 p_{i}(\overline{x}_i(k)) = \Xi \frac{1}{\overline{x}_{i}(k)} \frac{d} {d x} \Bigg \vert_{x = \overline{x}_{i}(k)}  f_i(\overline{x}_{i}(k)),
	\end{equation}	
where $\Xi$ is an appropriate scaling constant.
To better understand the consequences of having back-off probabilities determined by Equation (\ref{eq: probOpt}), one can imagine the following scenario: two agents compete to have access to a shared resource via AIMD. At capacity event each agent will back off with probabilities that depend on how much access to the shared resource they had in the past: if agent 1 had a larger share of the resource than agent 2, agent 1 will be more likely to back off and vice versa (weighted by their cost functions, $f_1(\cdot)$ and $f_2(\cdot)$).  

A natural extension to this idea comes from a number of applications  that have arisen recently. There, agents simultaneously compete for multiple resources and wish to do this in an optimal manner \cite{Corless2016}. In this case, the back off probabilities are coupled through a multivariate utility (i.e., some function of their share of each resource). For two resources $a$ and $b$, agents have utility functions $f_i: [0,1]^2 \to \R_+$, $(x_{ai},x_{bi}) \mapsto f_i(x_{ai},x_{bi}) $ with the same assumptions as before. At a capacity event $k_a$ associated with resource $a$, the probabilities now take the form:
	\begin{equation}
	 p_{i}(k_a) = \Xi \frac{1}{\overline{x}_{ai}(k_a)} \frac{\partial} {\partial x} \Bigg \vert_{x = \overline{x}_{ai}(k_a)}  f_i(\overline{x}_{ai}(k_a), \overline{x}_{bi}(k_a)),
	 \label{eq: probs}
	\end{equation}	
where the meaning of the quantities are as defined above and where a similar probabilistic update can be defined for resource $b$.\newline 

We wish now to establish that such coupled systems exhibit the ergodic property in the case of finite averaging. To facilitate our analysis we shall make a simplifying assumption in order to establish a contraction with a view to applying Theorem \ref{thm:barns}. As before, we shall assume that the state of all agents evolve in a stochastic manner according to the above drop policy \eqref{eq: probs}. 
However, we will assume finite-time averaging, i.e. we consider a fixed averaging time $N>0$ and averages for $c\in\{a,b\}$, $i=1,\ldots,n$ and $k_c\in\N$ are of the form
\begin{equation}
    \label{eq:finiteav}
    \tilde{x}_{ci}(k_c) = \frac{1}{N}\sum_{\ell = k_c-N+1}^{k_c} x_{ci}(\ell).
\end{equation}
Here, the summation is exclusively over capacity events of the respective resource. 
For the initial phase $0\leq k <N$, the average can be taken over a shorter period of time. As we are interested in long-term behaviour, this is of no relevance.

As the second new feature, we also assume that in our algorithm, after $M$ successive capacity events associated with a resource, agents stop evolving until they receive a {\em global event} notification. More specifically, a central agent sends a global event notification to all agents, when both the resources have experienced $M$ capacity events since the last global event notification.  We refer to this global event as a {\em meta} event. Agents respond to the meta event notification in a probabilistic way simultaneously by multiplicatively decreasing their demands by $\beta_1$ and $\beta_2$ or do not decrease; we call this as {\em global MD phase}. After the global MD phase, agents repeat this process up to the next $M$ capacity events of a resource. As we have mentioned, the objective of this construction is to create a situation by which a contraction for the overall system can be established in order to deduce ergodicity. The pseudocode relative to this procedure is described in Algorithm \ref{algo_multi11}.
\begin{rem}
To give a better intuition of the previous process,  observe that although each resource has different capacity events (and therefore different indices), there is a common notion of time shared between
all resources. We add a notion of ``waiting'' after M capacity events in order to enforce
a regularly occurring sequence of ``time instants'' at which the number of capacity events
for each resource is the same across all resources. This is needed, because one resource may
experience a higher frequency of capacity events than another and therefore complicate the analysis.
\end{rem}

\begin{algorithm}
\caption{Synchronised AIMD algorithm for multiple resources.} 
\label{algo_multi11}
\begin{algorithmic}[1]
   \State $k\gets 0$
   \State $M_a\gets 0$
   \State $M_b\gets 0$
   \While{Evolution is not interrupted}
   \For{$p \in \{a,b\}$}
   \If{Capacity Event for Resource $p$ and $M_p < M$}
       \State $M_p \leftarrow M_p +1$
       \State Agents enter the MD phase for resource $p$
    \Else 
    \If{$M_p = M$}
            \State Agents neither increase nor decrease their share of resource $p$
        \Else
            \State Agents increase their share of resource $p$ (the AI phase)
        \EndIf
    \EndIf
    \EndFor
    \State $k\leftarrow k + 1$
    \If{$M_a = M$ and $M_b = M$}
        \State $M_a \leftarrow 0$
        \State $M_b \leftarrow 0$
    \EndIf
	\EndWhile
\end{algorithmic}
\end{algorithm}

\subsection{A model of coupled resources and the main result}
We shall now construct a model of the system described in the previous section that can be used to establish ergodicity. More specifically, let us consider an AIMD network with $n$ agents and  two resources $a,b$ that evolve and synchronise at every meta event after $M=N$ capacity events. 

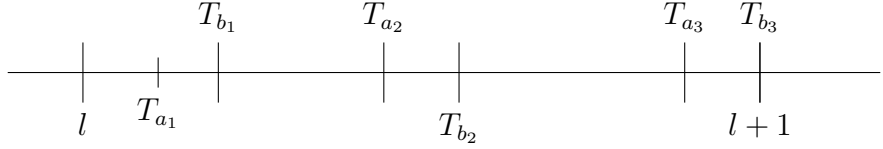
\begin{figure}[ht]  
\centering
\begin{tikzpicture}[scale=2]
\draw[black] (-3.5,0) -- (2.3,0);
\draw[black](-3,+0.2)--(-3, -0.2) node[below] {$l$};
\draw[black](-2.5,+0.1)--(-2.5, -0.1) node[below] {$T_{a_1}$};
\draw[black](-2.1, -0.2)--(-2.1,+0.2) node[above] {$T_{b_1}$};
\draw[black](-1, -0.2)--(-1,+0.2) node[above] {$T_{a_2}$};
\draw[black](-0.5,+0.2) --(-0.5, -0.2)node[below] {$T_{b_2}$};
\draw[black](1, -0.2)--(1,+0.2) node[above] {$T_{a_3}$};
\draw[black](1.5, 0.2)--(1.5,-0.2) node[below] {$l+1$};
\draw[black](1.5,-0.2)--(1.5, 0.2) node[above]{$T_{b_3}$};
\end{tikzpicture}
\caption{Evolution of capacity events of two resource $a$ and $b$. In the figure, $T_{i_j}$ denotes the time of the $j$-th capacity event of resource $i$.}
\label{fig:example}
\end{figure}

As an example, consider $N=3$, as depicted in Figure \ref{fig:example}. After $T_{a_3}$, resource $a$ waits until  resource $b$ reaches its third capacity event, before starting its evolution again. 

For resources $a,b$, let $x_a,x_b\in\Sigma$ denote the vectors of resource distributions among the different agents. 
We denote the sets of matrices describing the two different AIMD processes by $\mathcal{A} = \{A_j; j\in\mathcal{S}\}$, resp. $\mathcal{B} = \{B_j;j\in\mathcal{S}\}$. For $c\in\{a,b\}$, the two systems evolve according to
\begin{align}\label{eq:forced-synch}
x_c\left(k_{c}+1\right)&= A_c\left(k_{c}\right)x_c\left(k_{c}\right),
\end{align}
where \emph{$k_a,k_b\in\N$ denote the indices for capacity events of resource $a$ resp. $b$} and $A_a(k_a)\in\mathcal{A}, A_b(k_b)\in\mathcal{B}$ are chosen according to a probabilistic rule that will be described later. A priori, the time between capacity events depends on the specifics of the AIMD process and so the two resource processes are not synchronized. 

We will consider forced synchronization in the sense that there exists a constant $N\geq 1$ such that the two processes synchronize every $N$ steps. In practice, this means that if one of the processes has experienced $N$ capacity events first, it will wait until the other process also has achieved $N$ capacity events and at that time both processes will start again.

{So far, we have talked almost exclusively about a discrete-time process on the time scale of capacity events, but recall that in \eqref{eq: basic aimd} the variable $T(k)$ refers to continuous time. In order to be able to determine the sequence of capacity events for the different resources, we introduce, as a parenthesis, a continuous-time process evolving in actual time. However, it is important to note that in our results we only concern ourselves with the ergodicity of the discrete-time process.}

Now, we describe a model of forced synchronization.
For each resource, the time elapsed between two capacity events is determined by the state and the chosen AIMD matrix. This dependence is given by continuous functions
$T_a,T_b : \Sigma\times \mathcal{S}\to \R_+$. (We omit the specific form of $T_a,T_b$ as it is of no particular relevance here. It is easily derived.)
Let $\psi_a,\psi_b:\N\to \R_+$ be the functions that map the indices of capacity events to time. As a consequence of the synchronization every $N$ capacity events, the following relation holds true:
\begin{align}
\psi_a\left(lN\right)=\psi_b\left(lN\right), \quad l \in \N.
\end{align}
This is guaranteed by the following iterative definition: Set  $\psi_a(0)=\psi_b(0)=0$ and for $c\in \{a,b\}$, $l=0,1,2,,\ldots$
\begin{align}
\psi_c(lN + k+1 ) &:= \psi_c(lN+k) + T_c(x_c(lN+k),j_c(lN+k)), \quad \text{ for $k=0,\ldots,N-2$}, \nonumber \\
\tau_l &:= \max \left\{ \sum_{i=0}^{N-1} T_a(x_a(lN+i),j_a(lN+i)), \sum_{i=0}^{N-1} T_b(x_b(lN+i),j_b(lN+i))\right\}, \nonumber \\
\psi_a((l+1)N ) &:= \psi_b((l+1)N ) := 
\psi_a(lN ) + \tau_l. \label{eq:sync}
\end{align}

Given a sequence of matrices $\{A_a(k)\} \in \mathcal{A}^\N$, $\{A_b(k)\}\in \mathcal{B}^\N$, we define the corresponding transition operators $\Phi_a, \Phi_b$ for
$c\in\{a,b\}$ by
\begin{equation}\label{eq:transi-op}
\Phi_c(k,k) = I_n, \quad \Phi_c(m+1,k) = A_c(m+1)\Phi_c(m,k), \quad \forall m\geq k\in \N.
\end{equation}
Then, due to the synchronization guaranteed by (\ref{eq:sync}) at (what we call) \emph{meta-events}, the evolution of the states at meta-events $l\in 0,1,2,\dots$, is given by

\begin{equation}
\begin{bmatrix}
{x}_a((l+1)N) \\
{x}_b((l+1)N) 
\end{bmatrix}
=\begin{bmatrix}
\Phi_a((l+1)N, lN) & 0 \\
0 & \Phi_b((l+1)N,lN)
\end{bmatrix}\begin{bmatrix}
{x}_a(lN) \\
{x}_b(lN) 
\end{bmatrix}.
\label{eq: partial equations}
\end{equation}
Notice, however, that the previous equation is not sufficient to determine the evolution of the whole system as, in general, the matrices $A_c(k)$ and thus the matrices $\Phi_c((l+1)N, lN)$  are random matrices, whose values depend on the average allocation, $\tilde{x}_c$, over some window of capacity events. Accordingly, in this paper we will consider averages over finite time windows of length $N$. \newline

\emph{Remark:} There has been a simplifying choice that $N$, the length of the averaging time, equals $M$, the length of the time window of meta events. Also we assume that meta event windows for different resources are equal. This merely has the purpose of not over-encumbering the analysis and theorems with notation. The following results can be  extended to the more general case of time windows of length $M\neq N$ and $M_a\neq M_b$ for resources $a$ and $b$.

\subsection{Finite averaging}
Recall the definition of the averages $\tilde{x}_{ci}$, $c\in\{a,b\}, i=1,\ldots,n$, of length $N$ from \eqref{eq:finiteav}.
Let $\rho_{a,j}, j\in \mathcal{S}_a$ $\rho_{b,j}, j\in \mathcal{S}_b$ be probability functions as in \eqref{eq:probfuncs}. For example, they might be defined according to (\ref{eq: probs}).  The probabilistic law for $c\in \{a,b\}$ is 
\begin{align}
\mathbb P\left(  A_{c}(k) = A_{c,j} \ |\ \tilde{x}_{c}(k) \right)=\rho_{c,j}\left(\tilde{x}_a(k),\tilde{x}_b(k)\right).
\label{eq:probs2}
\end{align}
The probabilities in the interval $[lN-N+1, lN]$, determined by the averages $\tilde{x}_c(k)$, together with equations (\ref{eq: partial equations}) are sufficient  to determine the evolution of the meta-events.

Accordingly, we introduce the vector of partial averages over the interval $[lN-N+1, lN]$ as follows:
\begin{align}
z_{a}(l) &\triangleq \left[{x}_{a}(lN)^{\top} \quad \frac{1}{2}\sum_{i=0}^1 x_{a}(lN -i)^{\top} \quad \ldots \quad \frac{1}{N}\sum_{i=0}^{N-1} x_{a}(lN -i)^{\top} \right]^{\top},\nonumber\\
z_{b}(l) &\triangleq\left[{x}_{b}(lN)^{\top} \quad \frac{1}{2}\sum_{i=0}^1 x_{b}(lN -i)^{\top} \quad \ldots \quad \frac{1}{N}\sum_{i=0}^{N-1} x_{b}(lN -i)^{\top} \right]^{\top}.
\nonumber
\end{align}
Now we define for $l\in\N$
\begin{align}
\label{eq:highdimsys}
\zeta(l)= \left[z_{a}(l)^{\top}\quad z_{b}(l)^{\top}\right]^{\top} \in \R^{2nN}.
\end{align}
The evolution of $\zeta(l)$ is described by \begin{align}\label{eq:MC-FW}
\zeta(l+1)=\Gamma(l) \zeta(l),
\end{align}
where 
\begin{equation} \label{mat:Gamma}
\Gamma(l)\triangleq\begin{bmatrix}
\gamma_a(l) & 0 \\
0 & \gamma_b(l)
\end{bmatrix},
\end{equation}
and for $c\in \{a,b\}$
\begin{align}\label{eq:mat-gamma}
\gamma_c(l)=\begin{bmatrix}
\Phi_c((l+1)N,lN)&0&\ldots&0\\
\frac{1}{2}\sum_{i=0}^1\Phi_c((l+1)N-i,lN) & 0&\ldots &0\\
\frac{1}{3}\sum_{i=0}^2\Phi_c((l+1)N-i,lN)& 0 & \ldots&0\\
\vdots& \vdots& \vdots & \vdots\\
\frac{1}{N}\sum_{i=0}^{N-1}\Phi_c((l+1)N-i,lN)&0&\ldots& 0
\end{bmatrix} \in \R^{nN\times nN}.
\end{align}

Note that by construction $\Gamma(l)$ maps $\Sigma^{2N}$ into itself, so we may use these matrices to construct an equivalent IFS on this space. All that is required is to construct the appropriate probability functions. Of particular interest will be the matrix $\Gamma_1$, defined as the matrix obtained by choosing the full drop matrices $A_1,B_1$ in all instances. The corresponding matrices $\gamma_{1c}$, $c\in\{a,b\}$ are then of the form
(with $C_1$ equal to $A_1$ or $B_1$)
\begin{equation}
\label{eq:fulldropgamma}
    \gamma_{1c} = \begin{bmatrix}
C_1^N&0&\ldots&0\\
\frac{1}{2}\sum_{i=0}^1 C_1^{N-i} & 0&\ldots &0\\
\vdots& \vdots& \vdots & \vdots\\
\frac{1}{N}\sum_{i=0}^{N-1} C^{N-i}&0&\ldots& 0
\end{bmatrix} .
\end{equation}

The possible matrices $\Gamma$ occurring in \eqref{eq:MC-FW} are uniquely determined by the sequences $A_{a,j_0}, \ldots, A_{a,j_{N-1}}$ in $\mathcal{A}^{N}$ and $A_{b,\ell_0}, \ldots, A_{b,\ell_{N-1}}$ in $\mathcal{B}^{N}$. Thus, we require probability functions $p_{a,\nu},p_{b,\mu}:\Sigma^{2N}\to [0,1]$ for all $(\nu,\mu) \in \mathcal{S}_a^N\times\mathcal{S}_b^N$. Given the probability functions $\rho_{c,j}, c\in\{a,b\},j\in \mathcal{S}_c$ from \eqref{eq:probs2} the probability of a realization of the sequences $\nu=(j_0,\ldots,j_{N-1})$, $\mu=(\ell_0,\ldots,\ell_{N-1})$ can be computed recursively.
To this end,  let $z_{a,k}$, $k=1,\ldots,n$, denote the subvectors of $z_a\in \R^N$ and note that $(z_a,z_b)\in \Sigma_n^{2N}$ and the sequences $\nu,\mu$ determine the order in which capacity events occur for the two resources in a deterministic manner. Also note that using the entries of $\zeta$
we have for $k=0,\ldots,N-1$ that
\begin{equation}
\label{eq:avhelper}
   \tilde{x}_c(lN+k) = \frac{N-k}{N}z_{c,N-k}+\frac{1}{N}\sum_{i=0}^{k-1} \Phi_c(i+1,0) z_{c,1}. 
\end{equation}
From \eqref{eq:probs2} we thus obtain for the probabilities for $j_0,\ldots,j_{N-1}$
\begin{align}
\label{eq:highdimprobs}
    \mathbb{P}(j_0 =j| \zeta) 
    &= \rho_{a,j} (z_{a,N},z_{b,N}) \\
\intertext{and if the $k$th capacity event for resource $a$ occurs after exactly $m$ capacity events have occurred for resource $b$}
    \mathbb{P}(j_k = j | \zeta, j_0,\ldots,j_{k-1},\ell_0,\ldots,\ell_m) &=  \rho_{a,j} 
    \left( \tilde{x}_a(lN+k), \tilde{x}_b(lN+m) \right) , 
    \label{eq:highdimprobs2}
\end{align}
with a similar formula for the probabilities for $\mu$.
Note that the argument of the function on the right can be computed using the knowledge of $\zeta$ and the indices $j_0,\ldots,j_{k-1},\ell_0,\ldots,\ell_m$ with the help of \eqref{eq:avhelper}. The definition of the functions $p_{\nu,\mu} = p_{a,\nu}p_{b,\mu}: \Sigma^{2N}\to [0,1]$, $(\nu,\mu)\in \mathcal{S}_a^N\times \mathcal{S}_b^N$ now just uses standard manipulations. We omit the details.

\subsection{Ergodicity of the Synchronised AIMD}

We are now ready to present the main results of this paper. The conditions of Theorem~\ref{thm:barns} require a notion of average contractivity on the state space. If we want to apply this to the IFS defined by 
\eqref{eq:highdimsys} with the probabilities 
described in \eqref{eq:highdimprobs}, we need to show that the matrices in \eqref{eq:mat-gamma} have a contractivity property. To this end, we introduce on $\R^{nN}$ the norm
\begin{equation}
    \left\| \begin{bmatrix}
     x_1^\top, \ldots, x_N^\top 
    \end{bmatrix}^\top \right\|_{N,1} \triangleq \max_{i=1,\ldots,N} \|x_i\|_1 , \quad x = \begin{bmatrix}
     x_1^\top, \ldots, x_N^\top 
    \end{bmatrix}^\top \in \R^{nN}, x_i\in\R^n, i=1,\ldots,N.
\end{equation}

We denote by $\mathcal{A}_\Gamma$ the set of matrices that can be constructed from the set $\mathcal{A},\mathcal{B}$ in the form \eqref{eq:mat-gamma}. To start, consider the following lemma:

\begin{lem}\label{lem:condi-a-b-c}
\begin{itemize}
\item[a)]For all $\Gamma \in \mathcal{A}_\Gamma$ and $\zeta\in \mathbb R^{2nN}$, we have 
\begin{align}
\left \| \Gamma \zeta\right\|_{N,1}\le \left\|\zeta\right\|_{N,1}.
\end{align}
\item[b)] The subspace 
\begin{align}
\mathcal{E}= \mathrm{ker} \begin{bmatrix}
e^\top & 0_{(N-1)n} & 0_n & 0_{(N-1)n} \\ 0_n  & 0_{(N-1)n} &e^\top &0_{(N-1)n}
\end{bmatrix} =\left\{\zeta=\left[z_{a}^{\top}\quad z_{b}^{\top}\right]^{\top} \in \R^{2nN}: e^\top z_{a,1}=e^\top z_{b,1} = 0\right\}  
\end{align}
is invariant under all $\Gamma\in \mathcal{A}_\Gamma$.
\item[c)]  Recall $\Gamma_1\in \mathcal{A}_\Gamma$ from \eqref{eq:fulldropgamma}. Then, for all $\zeta \in \mathcal{E}$
\begin{align}
 \left\|\Gamma_1\zeta\right\|_{N,1}\le q\left\|\zeta\right\|_{N,1}   ,
\end{align}
where
\begin{align}
q= \text{max}\left(\frac{1}{N}\sum_{i=1}^{N} \beta_a^i,\frac{1}{N}\sum_{j=1}^{N} \beta_b^j \right)<1,
\end{align}
and where $\beta_a$ and $\beta_b$ are the multiplicative-decrease parameters for resource $a$ and resource $b$, respectively.
\end{itemize}
\end{lem}
\begin{proof}
a) The block matrices in the first block column of $\Gamma$ are convex combinations of products of column stochastic matrices and, therefore, column stochastic, see  \eqref{eq:mat-gamma}. In particular, the induced $1$-norm of these matrices is equal to $1$. 
This shows
\begin{equation*}
 \left \| \Gamma \zeta\right\|_{N,1}\le
\max_{c=a,b} \max_{k=1,\ldots,N} \left \| \frac{1}{k}\left(\sum_{i=0}^{k-1}\Phi_c((l+1)N-i,lN)\right) z_{c,1}\right\|_{1}
 \leq \max \{ \|z_{a,1}\|_1, \|z_{b,1}\|_1\} \leq
 \left\|\zeta\right\|_{N,1}.
\end{equation*}
b) The upper left $n\times n$ block of a matrix $\gamma_a$ is the product of column stochastic matrices and, therefore, column stochastic. For a column stochastic matrix $M$, it is easy to see that $e^\top x = 0$ implies $e^\top M x = e^\top x = 0$. The same argument applies to the matrices $\gamma_b$.

c) Recall that we denote by $A_1\in\mathcal{A}$ the complete drop matrix for resource $a$. From \cite[Lemma 3.5]{Corless2016} it is known that $e^\top x =0$ implies $\|Ax\|_1 \leq \beta \|x\|_1$ as well as $e^\top Ax =0$. For the block entries of the first row of $\gamma_{1a}$ this implies 
\begin{equation*}
    \left\| \frac{1}{k}\left(\sum_{i=0}^{k-1} A_1^{N-i}\right) x\right\|_{1} 
    \leq \frac{1}{k} \left(\sum_{i=0}^{k-1} \beta_a^{N-i}\right) \|x\|_1.
\end{equation*}
As $\beta_a<1$, the constant on the right is maximized for $k=N$ (over the options $k=1,\ldots,N$). The claim now follows by rearranging the sum and maximizing with respect to $a,b$.
\end{proof}

\textcolor{black}{This part of the lemma shows that the iteration of random choices of $\Gamma$ matrices are contractive when studied with respect to a suitable norm. This is crucial for establishing existence of a unique invariant and attractive measure for the associated Markov chain.}

\begin{thm}
Consider the AIMD algorithm with finite averaging and 
Lipschitz continuous probability functions $\rho_{a,j}, j\in \mathcal{S}_a$, $\rho_{b,j}, j\in \mathcal{S}_b$. Assume that the probabilities $\rho_{a,1}, \rho_{b,1}$ of the full drop matrices for both resources are strictly positive on $\Sigma^{2}_n$.
Then for all $N\geq 1$, the IFS \eqref{eq:MC-FW} associated to finite averaging over length $N$ with probability functions given by \eqref{eq:highdimprobs}, \eqref{eq:highdimprobs2} has the following properties:
There exists a unique invariant and attractive measure $\pi^{2N}$ on $\Sigma^{2N}$. Furthermore, for all initial conditions $\zeta(0)\in \Sigma^{2N}$, we have almost surely
\begin{align}
\lim_{k\to \infty} \frac{1}{k+1} \sum_{j=0}^{k} \zeta(j)= \int_{\Sigma^{2N}} \zeta \mathrm{d}\pi^{2N}(\zeta)= \mathbb E_{\pi^{2N}}\left[\zeta\right].
\end{align}
\end{thm}
\begin{proof}
First note that the probability functions $p_{(\mu,\nu)}$, $(\mu,\nu) \in \mathcal{S}_a^N \times \mathcal{S}_b^N$ constructed using \eqref{eq:highdimprobs} and \eqref{eq:highdimprobs2} for the lifted system are obtained as products of the functions $\rho_{a,j}, j\in \mathcal{S}_a$ $\rho_{b,j}, j\in \mathcal{S}_b$. Therefore, they are also Lipschitz continuous. In addition, as $\rho_{a,1}, \rho_{b,1}$ are strictly positive on $\Sigma^{2}_n$, there exists a constant $\hat{p}>0$, such that for the probability function $p_1$ associated with $\Gamma_1$ we have $p_1(\zeta) >\hat{p}$ for all $\zeta\in \Sigma^{2N}_n$.

From now on, we abbreviate $\kappa=(\mu, \nu)$.
In order to apply 
Theorem~\ref{thm:barns}, we must show that the sufficiency conditions stated there are satisfied, i.e., that there exist $r<1$ and $\delta>0$ such that for all $\zeta,\eta \in \Sigma^{2N}$, 
\begin{align}\label{eq:contra}
\sum_{\kappa\in \mathcal{S}_a^N \times \mathcal{S}_b^N}p_{\kappa}(\zeta)\left \|\Gamma_\kappa(\zeta-\eta)\right\|_{N,1}\le r \|\zeta-\eta\|_{N,1},
\end{align}
and, for $\mathcal{C}(\zeta,\eta) = \{\kappa \in  \mathcal{S}_a^N \times \mathcal{S}_b^N;\left\|\Gamma_\kappa(\zeta-\eta)\right\|_{N,1}\le r \|\zeta-\eta\|_{N,1} \} $ we require
\begin{align}\label{eq:prob-condi}
\sum_{\kappa \in \mathcal{C}(\zeta,\eta)} p_{\kappa}(\zeta)p_{\kappa}(\eta) \ge \delta.
\end{align}
In order to show that \eqref{eq:contra} holds, first notice that when $\zeta,\eta\in \Sigma^{2N}$, we have $\zeta-\eta\in \mathcal{E}$. Hence, Lemma \ref{lem:condi-a-b-c} (c) implies the existence of a constant $q<1$ such that
\begin{align*}
\left \|\Gamma_1(\zeta-\eta)\right\|_{N,1}\le q \|\zeta-\eta\|_{N,1}.
\end{align*}
Using the existence of the constant $\hat{p}>0$ and Lemma \ref{lem:condi-a-b-c} (a), we obtain that for all $\zeta,\eta \in \Sigma^{2N}$
\begin{align}
\sum_\kappa p_{\kappa}(\zeta)\left \|\Gamma_\kappa(\zeta-\eta)\right\|_{N,1}\le p_{1}(\zeta)q+ (1- p_{1}(\zeta))\le \hat{p}q + (1- \hat{p}) \triangleq r <1.
\end{align}
To demonstrate \eqref{eq:prob-condi}, note that $q\le r$; hence
\begin{equation*}
\sum_{\kappa \in \mathcal{C}(x,y)} p_{\kappa}(x)p_{\kappa}(y)  \ge p_{1}(y)p_{1}(x)\ge \hat{p}^{2}>0.   
\end{equation*}
The claim now follows from an application of Theorem~\ref{thm:barns}.
\label{thm:main result}
\end{proof}

{\bf Comment :} {Before proceeding, it is important to place the above theorem in context. Establishing ergodiciy is important for a number of reasons. At a very basic level, it means that results, from a probabilistic perspective, are independent of initial conditions. This is important from the perspective of developing distributed optimization algorithms; a basic feature of such algorithms is that they should be uniformly convergent irrespective of initial conditions. An additional implication of ergodicity is that, when exploring such algorithms from a practical perspective using simulations, rigorous conclusions can be drawn from these simulations. Finally, from a very practical perspective, ergodicity is an essential feature of any algorithm that underpins economic contracts, in which metrics such as Quality of Service are important and need to be independent of initial conditions, such as in the sharing economy.}\newline 

\section{Simulations}
To illustrate the main result, we now include the following simulations. The objective is to demonstrate the independence of initial conditions with respect to a number of moments for a given system. More specifically, we present 150 Monte Carlo simulations with 4 agents, 2 resources (both with capacity 1), and initial conditions randomly chosen in $[0, 0.25]$. The network parameters are described in Table \ref{tab: param}.

\begin{table}[ht!]
    \centering
    \begin{tabular}{c|c|c|c|c}
        Agent & $\alpha_a$ & $\beta_a$ & $\alpha_b$ & $\beta_b$  \\
         \hline 
         \hline
        1 & 0.01 & 0.95 & 0.07 & 0.65 \\
        2 & 0.08 & 0.9 & 0.08 & 0.7 \\
        3 & 0.61 & 0.85 & 0.025 & 0.8 \\
        4 & 0.045 & 0.75 & 0.02 & 0.85 \\
    \end{tabular}
    \caption{Parameters for the 4 agents and the 2 resources.}
    \label{tab: param}
\end{table}

At each capacity event, each agent  multiplicatively decreases their share of resource $a$ or $b$ with probability 
\begin{equation}
    p_i(k) = \frac{1}{2N}\left(\sum_{\ell=0}^N\tilde{x}_{a,i}(k-\ell) +  \sum_{j=0}^N\tilde{x}_{b,i}(k-j) \right).
\end{equation}
Figure \ref{fig1-avg} shows one realization of the system. As can be seen, when resource $a$ (the blue curve) reaches the $5$\textsuperscript{th} capacity event in a row, it stops evolving, and it waits for resource $b$ (the green line) to reach the $5$\textsuperscript{th} capacity event. Once that happens, the two resources are again synchronised and the evolution starts again. 

Figure \ref{fig2}
present the evolution of the mean and variance (over 150 Monte Carlo simulations) of allocations of resources $a$ and $b$ over time. As can be seen, despite the difference in initial conditions, every agent eventually converges to the same mean and variance for both resources, as predicted by Theorem \ref{thm:main result}.

\begin{figure}[t]
\centering
	\includegraphics[width=0.45\textwidth]{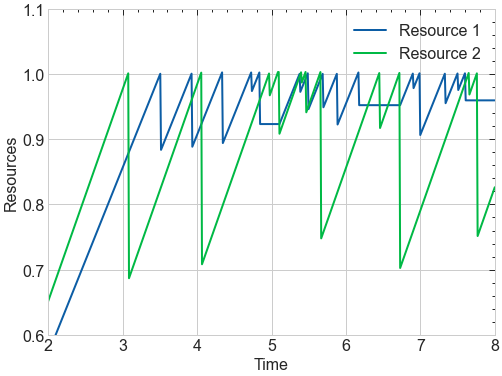}
	\caption{The evolution of the utilization of two resources under the AIMD algorithm with $N=5$.}
	\label{fig1-avg} 
\end{figure}

\begin{figure}[h!]
	\centering
\begin{subfigure}[t]{0.45\textwidth}
	\centering
	\includegraphics[width=\textwidth]{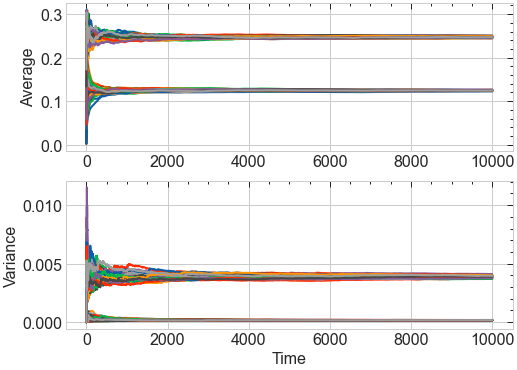}
	\caption{Agent 1}
	\label{fig2-gradGE} 
\end{subfigure}
\begin{subfigure}[t]{0.45\textwidth}
	\centering
	\includegraphics[width=\textwidth]{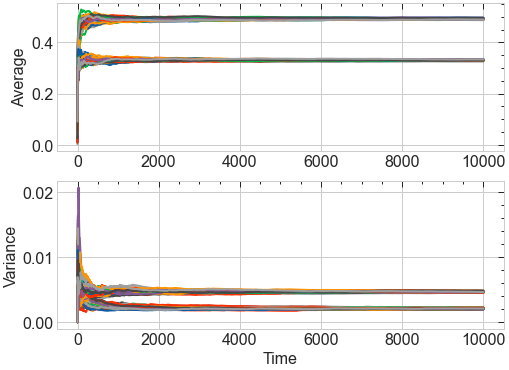}
	\caption{Agent 2}
	\label{fig2-gradGE2} 
\end{subfigure}
\\

\begin{subfigure}[t]{0.45\textwidth}
	\centering
	\includegraphics[width=\textwidth]{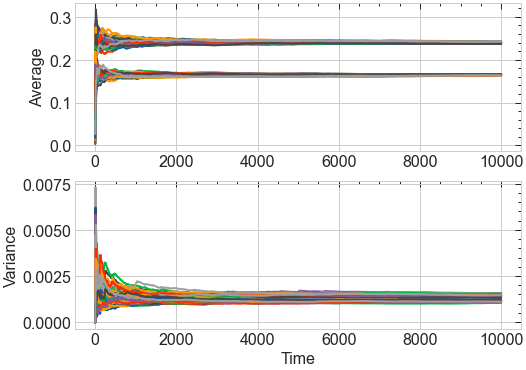}
	\caption{Agent 3}
	\label{fig2-gradGE3} 
\end{subfigure}
\begin{subfigure}[t]{0.45\textwidth}
	\centering
	\includegraphics[width=\textwidth]{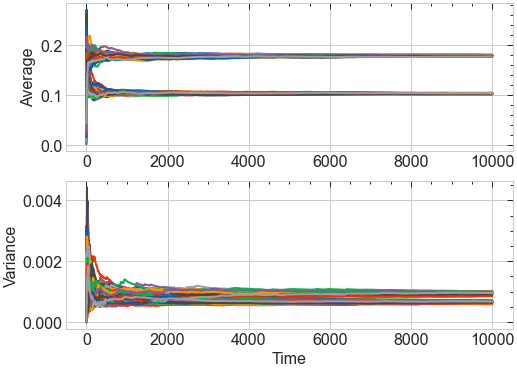}
\caption{Agent 4}
	\label{fig2-gradGE4} 
\end{subfigure}
	\caption{The evolution of the mean and variance (across 150 Monte Carlo simulations) of the allocation $x$ for the same two resources and four agents as in Figure \ref{fig1-avg}.}
	\label{fig2}
\end{figure}

\section{Conclusions and Future Work}

In this work, we explored the behaviour of two AIMD networks that regularly synchronise with one another. We proved that such a system inherits the ergodic property of the two individual AIMD networks and this result was showcased through extensive Monte Carlo simulations. This result will be the basis for an extension of this analysis to a broader class of interconnected AIMD networks.

\FloatBarrier
\bibliographystyle{plain}
\bibliography{sample}

\begin{thebibliography}{10}

\bibitem{Syed2020}
S.~E. Alam, F.~Wirth, J.~Y. Yu, and R.~Shorten.
\newblock The convergence of finite-averaging of {AIMD} for distributed
  heterogeneous resource allocations.
\newblock {\em arXiv:2001.08083 [math.OC]}, 2020.

\bibitem{Barnsley1988}
M.~F. Barnsley, S.~G. Demko, J.~H. Elton, and J.~S. Geronimo.
\newblock Invariant measures for {M}arkov processes arising from iterated
  function systems with place-dependent probabilities.
\newblock {\em Annales de l'I.H.P. Probabilités et statistiques},
  24(3):367--394, 1988.

\bibitem{Chiu1989}
D.~Chiu and R.~Jain.
\newblock Analysis of the increase and decrease algorithms for congestion
  avoidance in computer networks.
\newblock {\em Computer Networks and ISDN Systems}, 17(1):1--14, 1989.

\bibitem{Corless2016}
M.~Corless, C.~King, R.~Shorten, and F.~Wirth.
\newblock {\em {AIMD} Dynamics and Distributed Resource Allocation}.
\newblock Number~29 in Advances in Design and Control. SIAM, Philadelphia, PA,
  2016.

\bibitem{Crisostomi2020}
Emanuele Crisostomi, Bissan Ghaddar, Florian H{\"a}usler, Joe Naoum-Sawaya,
  Giovanni Russo, and Robert Shorten, editors.
\newblock {\em Analytics for the Sharing Economy: Mathematics, Engineering and
  Business Perspectives}.
\newblock Springer Nature, 2020.

\bibitem{Fioravanti2019}
Andre~R Fioravanti, Jakub Mare{\v{c}}ek, Robert~N Shorten, Matheus Souza, and
  Fabian~R Wirth.
\newblock On the ergodic control of ensembles.
\newblock {\em Automatica}, 108:108483, 2019.

\bibitem{Jacobson1988}
V.~Jacobson.
\newblock Congestion avoidance and control.
\newblock {\em {SIGCOMM} Comput. Commun. Rev.}, 18(4):314--329, Aug. 1988.

\bibitem{Marecek2015}
Jakub Marecek, Robert Shorten, and Jia~Yuan Yu.
\newblock Signaling and obfuscation for congestion control.
\newblock {\em International Journal of Control}, 88(10):2086--2096, 2015.

\bibitem{Marecek2016}
Jakub Marecek, Robert Shorten, and Jia~Yuan Yu.
\newblock r-extreme signalling for congestion control.
\newblock {\em International Journal of Control}, 89(1):1972--1984, 2016.

\bibitem{Wirth2019}
F.~Wirth, S.~St{\"u}dli, J.~Y. Yu, M.~Corless, and R.~Shorten.
\newblock Nonhomogeneous place-dependent {M}arkov chains, unsynchronised
  {AIMD}, and optimisation.
\newblock {\em J. ACM}, 66(4):24:1--24:37, 2019.

\end{thebibliography}
\end{document}